\documentclass[12pt,reqno,,a4paper,twoside]{amsart}
\usepackage{amssymb,amsmath,amscd,amstext,amsthm,amsfonts}
\usepackage[usenames]{color}
\usepackage{graphicx}
\usepackage{amsmath}
\usepackage{amssymb}
\usepackage{amsfonts}
\usepackage{multicol}
\usepackage[ansinew]{inputenc} 
\usepackage{amscd} 
\usepackage[mathscr]{eucal}

\newcommand{\Scal}{\mathcal{S}} 
\newcommand{\Mat}{{\rm Mat}}

\newcommand{\R}{\mathbb{R}}

\newcommand{\F}{\mathcal{F}}

\newcommand{\blV}{\bullet}
\newcommand{\whV}{\circ}
\newcommand{\plus}{\oplus}

\newcommand{\bbEdge}{\bullet\!\!-\!\!\bullet}
\newcommand{\wwEdge}{\circ\!\!-\!\!\circ}
\newcommand{\wbEdge}{\circ\!\!-\!\!\bullet}
\newcommand{\SD}{{\rm SD}}
\newcommand{\rank}{{\rm rank}}
\newcommand{\Null}{{\rm Ker}}
\newcommand{\Red}{\Rscr_\ast}

\newcommand{\Ascr}{\mathscr{A}}
\newcommand{\Vscr}{\mathscr{V}}
\newcommand{\Rscr}{\mathscr{R}}

\newcommand{\Fscr}{\mathscr{F}}

\newtheorem{teor}{Theorem}[section]
\newtheorem{rmk}[teor]{Remark}

\newtheorem{lema}[teor]{Lemma}
\newtheorem{prop}[teor]{Proposition}
\newtheorem{coro}[teor]{Corollary}
\newtheorem{defi}[teor]{Definition}

\begin{document}

\title[Rank of Stably Dissipative Graphs]{Rank of Stably Dissipative Graphs}

\author[P.~Duarte]{Pedro Duarte}
\address{Departamento de Matemática and Cmaf \\
Faculdade de Ciências\\
Universidade de Lisboa\\
Campo Grande, Edifício C6, Piso 2\\
1749-016 Lisboa, Portugal 
}
\email{pduarte@ptmat.fc.ul.pt}

\author[T.~Peixe]{Telmo Peixe}
\address{Departamento de Matemática and Cmaf \\
Faculdade de Ciências\\
Universidade de Lisboa\\
Campo Grande, Edifício C6, Piso 2\\
1749-016 Lisboa, Portugal 
}
\email{tjpeixe@fc.ul.pt}

\date{\today}

\begin{abstract}
For the class of stably dissipative Lotka-Volterra systems we prove that the rank
of its defining matrix, which is the dimension of the associated invariant foliation,
is completely determined by the system's graph.
\end{abstract}

\maketitle

\section{Introduction}

In his work ``Leçons sur la Théorie Mathématique de la Lutte pour la Vie'' \cite{Vo1931} Volterra
began the study of differential equations

\begin{equation}\label{eq1LV}
\dot{x}_i(t)=x_i(t)\Bigg(r_i+\sum_{j=1}^n a_{ij}x_j(t)\Bigg), \quad i=1,\dots,n,
\end{equation}
where $x_i(t)\ge 0$ represents the density of population $i$ in time $t$ and $r_i$ its intrinsic rate of growth or decay. Each coefficient $a_{ij}$ represents the effect of population $j$ on population $i$. If $a_{ij}>0$ this means that  population $i$ benefits from population $j$. $A=(a_{ij})$ is said to be the interaction matrix of the system (\ref{eq1LV}).
This system of differential equations is usually referred to as the Lotka-Volterra equations.

For two-dimensional 
systems we can completely analyze the dynamics of~(\ref{eq1LV}) (see, for example, \cite{HS1998}). 
However we are far from being capable of doing the same for higher dimensional Lotka-Volterra systems, 
in spite of some important results \cite{RW1984,Ta1996,DFO1998,ZL2010}.

There is a close relation between studying the dynamical properties of a Lotka-Volterra system and the algebraic properties of its interaction matrix, and depending on that these systems can be classified in tree main classes: \textit{cooperative} or \textit{competitive}; \textit{conservative}; and \textit{dissipative}.

For the {\em cooperative} and {\em competitive} systems , overall results were obtained by Smale \cite{Sm1976} and Hirsch \cite{Hi1982, Hi1988}, among others, for example, Zeeman \cite{Ze1993, Ze1995, Ze1996}, Van Den Driessche \textit{et al.} \cite {DZ1998}, Hofbauer \textit{et al.} \cite{HS1994}, Smith \cite{Sm1986} and Karakostas \textit{et al.} \cite {KG1988}.

Concerning the {\em conservative} systems, the initial investigations are attributed to Volterra,
who also defined the class of {\em dissipative}
 systems \cite{Vo1931} looking for a generalization of the predator-prey system.
Given  $A\in\Mat_{n\times n}(\R)$ and a point $q\in\R^n_+$
we can write  the  Lotka-Volterra system (\ref{eq1LV}) as
\begin{equation}\label{LV}
 \frac{dx}{dt}=X_{A,q}(x)\;,
\end{equation}
where $X_{A,q}(x)=x \ast A\,(x-q)$. The symbol `$\ast$' denotes point-wise multiplication of vectors in $\R^n$.
We say that system~(\ref{LV}), the matrix $A$, or the vector field $X_{A,q}$,  are {\em dissipative} \, iff\, there is a positive diagonal matrix $D$ such that
$Q(x)=x^T A D x\leq 0$ for every $x\in \R^n$. 
Notice this condition is equivalent to $x^T D^{-1} A x\le 0$, because
$$x^T D^{-1} A x = (D^{-1} x)^T A D (D^{-1} x) = Q(D^{-1} x) \le 0 \;.$$
When $A$ is dissipative, ~(\ref{LV}) admits the Lyapunov function
\begin{equation}\label{Liapunov_function}
h(x)=\sum_{i=1}^n \frac{ x_i-q_i\,\log x_i}{d_i} \;,
\end{equation}
which decreases along orbits of $X_{A,q}$.
In fact the derivative of $h$ along orbits of $X_{A,q}$ is 
\begin{equation}\label{Liapunov_function_derivative}
\dot{h}=\sum_{i,j=1}^n \frac{a_{ij}}{d_i}(x_i-q_i)(x_j-q_j)
=(x-q)^T D^{-1} A (x-q) \le 0.
\end{equation}
Since $h$ function is proper, $X_{A,q}$ determines a complete semi-flow 
$\phi_{A,q}^t:\R^n_+\hookleftarrow$, defined for all $t\geq 0$.
Let  $\Gamma_{A,q}$ denote the forward limit set of~(\ref{LV}),
i.e., the set of all accumulation points of  $\phi_{A,q}^t(x)$ as $t\to+\infty$,
sometimes referred to as the system's attractor.
By La Salle's theorem \cite{La1968}  we know that for dissipative systems
 $\Gamma_{A,q}\subseteq \{\,\dot{h}=0\,\}$.

The notion of \textit{stably dissipative} is due to Redheffer \textit{et al.}, who in a series of papers in the $80$'s \cite{RZ1981,RZ1982,RW1984,Re1985,Re1989} studied the asymptotic stability of this class of systems, using the term \textit{stably admissible} systems. Redheffer \textit{et al.} designated by \textit{admissible} the matrices that Volterra had initially classified as \textit{dissipative} \cite{Vo1931}.
We now give the precise definition of stably dissipative system.
Given a matrix $A\in\Mat_{n\times n}(\R)$ we say that another real matrix $\widetilde{A}\in\Mat_{n\times n}(\R)$ is a {\em perturbation of $A$} iff
\begin{displaymath}
\tilde{a}_{ij}=0 \, \Leftrightarrow \, a_{ij}=0.
\end{displaymath}
We say that a given matrix $A$, $X_{A,q}$, or~(\ref{LV}), is {\em stably dissipative} 
 iff any sufficiently small perturbation $\widetilde{A}$ of $A$ is dissipative, i.e.,
\begin{displaymath}
\exists \, \epsilon>0 \, : \, \max_{i,j}|a_{ij}-\tilde{a}_{ij}|<\epsilon \,\, \Rightarrow \,\, \widetilde{A} \textrm{ is dissipative}.
\end{displaymath}

\bigskip

From the interaction matrix $A$ we can construct a graph $G_A$ 
having as vertices the $n$ species $\{1,\ldots, n\}$. See definition \ref{graph} below. 
An edge is drawn connecting the vertices $i$ and $j$ whenever $a_{ij}\neq 0$ or $a_{ji}\neq 0$.
 Redheffer \textit{et al.} \cite{RZ1981,RZ1982,RW1984,Re1985,Re1989} have characterized the class
of stably dissipative systems and its attractor $\Gamma_{A,q}$ in terms of the graph $G_A$.
In particular, they describe a simple reduction algorithm, running on the graph $G_A$,
that `deduces' every restriction of the form $\Gamma_{A,q}\subseteq \{ x\,:\, x_i=q_i\}$,
  $1\leq i\leq  n$, that holds for every stably dissipative system with interaction graph $G_A$.
To start this algorithm  they use the following

\begin{lema}\label{lema2.1.ZL2010}
Given a stably dissipative matrix $A$, there is some positive
diagonal matrix $D$ such that
$$ \sum_{i,j=1}^n \frac{a_{ij}}{d_i} \,w_i\,w_j = 0 \quad \Rightarrow\quad a_{ii}\,w_i=0,\; \forall \, i=1,\ldots, n\;.$$
\end{lema}

Since by La Salle's theorem
$$\Gamma_{A,q}\subseteq \left\{\, x\in\R^n_+\,:\, \sum_{i,j=1}^n \frac{a_{ij}}{d_i}\,(x_i-q_i)\,(x_j-q_j)=0\,\right\}\;,$$
it follows that $\Gamma_{A,q}\subseteq \{ x\,:\, x_i=q_i\}$ for every $i=1,\ldots, n$ such that
$a_{ii}<0$. A vertex $i$ is colored black, $\blV$, to state that  $\Gamma_{A,q}\subseteq \{ x\,:\, x_i=q_i\}$ holds.
Similarly, a cross is drawn at a vertex $i$, $\plus$, to state that 
$\Gamma_{A,q}\subseteq \{ x\,:\, X_{A,q}^i(x)=0 \}$, which means   
$\{x_i={\rm const}\}$ is an  invariant foliation under  $\phi_{A,q}^t:\Gamma_{A,q}\hookleftarrow$.
All other vertices are coloured white, $\whV$.
Before starting their procedure, as $a_{ii}\leq 0$ for all $i$,
 they colour in black every vertex $i\in\{1,\ldots, n\}$ such that
$a_{ii}<0$, and in white all other vertices. This should be interpreted as a collection of statements about the attractor $\Gamma_{A,q}$.
The reduction procedure consists of  the following  rules,
  corresponding to valid inference rules:
\begin{itemize}
	\item[(a)] If  $j$ is a $\bullet$ or $\oplus$-vertex and all of its neighbours are $\bullet$, except for one vertex $l$, then   colour $l$ as $\bullet$;
	\item[(b)] If  $j$ is a $\bullet$ or $\oplus$-vertex and all of its neighbours are $\bullet$ or $\oplus$, except for one vertex $l$, then draw $\oplus$ at the vertex $l$;
	\item[(c)] If  $j$ is a $\circ$-vertex and all of is neighbours are $\bullet$ or $\oplus$, then   draw $\oplus$ at the vertex $j$.
\end{itemize}
 Redheffer \textit{et al.}  define the  {\em reduced graph} of the system,   $\Rscr(A)$, 
as the graph  obtained from $G_A$ by successive applications of the reduction rules  (a), (b) and (c) 
 until they can no longer be applied.
In~ \cite{RW1984} Redheffer and Walter proved the following result, which in a sense states that the previous algorithm on $G_A$
can not be improved.

\begin{teor}\label{Theorem_1_Redheffer_Walter}
Given a  stably dissipative matrix $A$,
\begin{itemize}
	\item[(a)] If $\Rscr(A)$ has only $\blV$-vertices 
	then $A$ is nonsingular, the stationary point $q$ is unique and every solution of (\ref{LV}) converges,  as $t\to\infty$, to  $q$.
	\item[(b)] If $\Rscr(A)$ has only $\blV$ and $\plus$-vertices, but not all $\blV$,
	  then $A$ is singular, the stationary point $q$ is not unique, and every solution of (\ref{LV}) has a limit, as $t\to\infty$, that depends on 
	  the initial condition.
	\item[(c)] If $\Rscr(A)$ has at least one $\whV$-vertex then there exists a stably dissipative matrix $\widetilde{A}$, with $G_{\widetilde{A}}=G_A$, such that the system (\ref{LV}) associated with $\widetilde{A}$ has a nonconstant periodic solution.
\end{itemize}
\end{teor}

In a very recent paper, Zhao and Luo \cite{ZL2010} gave necessary and sufficient conditions for a matrix to be stably dissipative, see  proposition~\ref{sd:char}.

\bigskip

The previous theorem tells us that when $\Rscr(A)$ has  only $\blV$-vertices, then the matrix $A$ has always maximal rank, $\rank(A)=n$.
In case $\Rscr(A)$ has only $\blV$ and $\plus$-vertices, then rank of the matrix $A$ equals the dimension of an invariant foliation.
This led us to establish that the rank of any stably dissipative matrix  only depends on its graph.
In particular the same is true for the  dimension of the invariant foliation of any stably dissipative system.
See theorem~\ref{main} and corollary~\ref{main_coro} of section~\ref{main:results}.
 
\bigskip
\bigskip

\section{ Dissipative Systems and Invariant Foliations }

Assume the Lotka-Volterra field $X_{A,q}(x)=x*A(x-q)$, defined in~(\ref{LV}),
is associated with a dissipative matrix $A$ with $rank(A)=k$. Let $W$ be a $(n-k)\times n$ matrix  whose rows
 form a basis of
\begin{displaymath}
\Null(A^T)=\{\,x\in\R^n\,:\,x^T A=0\,\}\,.
\end{displaymath}

Define the map $g:\R_+^n \to \R^{n-k}$, $g(x)=W\log x$, where
$\log x=(\log x_1,\dots,\log x_n)$.
Denoting by $D_x={\rm diag}(x_1,\ldots, x_n)$ the diagonal matrix, the jacobian matrix of $g$,
$D g_x= W\,D_x^{-1}$, has maximal rank $n-k$. Hence  $g:\R_+^n \to \R^{n-k}$
is a submersion. 

\bigskip

We denote by $\Fscr_A$ the pre-image foliation, whose leaves are the 
the pre-images $g^{-1}(c)$ of $g$. By a classical theorem on Differential Geometry,
each non-empty pre-image $g^{-1}(c)$ is a submanifold of dimension $k$.
Recall that the dimension of a foliation is the common dimension
of its leaves. A foliation $\Fscr$ is said to be invariant under a vector field $X$
if $X(x)\in T_x\F$ for every $x$, where $T_x\F$ denotes the tangent space at $x$ to the unique
leaf of $\F$ through $x$. This condition is equivalent to say that the flow of $X$ preserves
the leaves of $\F$.

\begin{prop}\label{inv:foliation}
If $A$ is dissipative then the foliation $\F_A$ is $X_{A,q}$-invariant with $\dim(\F_A)=rank(A)$.
\end{prop}

\bigskip

\begin{proof}
We have
\begin{align*}
Dg_x(X_{A,q}(x)) & =   Dg_x(D_xA(x-q))  \\
 					 & =   WD_x^{-1}D_xA(x-q)  \\
 					 & =   WA(x-q)=0  \;,
\end{align*}
because $WA=0$. Hence  
 $X_{A,q}(x)\in T_x\F_A$ and $\F_A$ is $X_{A,q}$-invariant.
\end{proof}

\bigskip

Defining the {\em symmetric} and {\em skew-symmetric parts} of a matrix $A$ by
$$ A^{\rm sym}=\frac{A+A^T}{2} \quad \text{ and } \quad A^{\rm skew} =\frac{A-A^T}{2}\;, $$
the following decompositions hold
$$ A= A^{\rm sym} + A^{\rm skew} \quad \text{ and } \quad
A^T= A^{\rm sym} - A^{\rm skew}\;.$$

\bigskip

The following lemma is a simple but key observation

\begin{lema}\label{kernel_equality}
Given   a dissipative matrix $A$,
\begin{displaymath}
\Null(A)=\Null(A^T)=\Null(A^{\rm sym})\cap \Null(A^{\rm skew}).
\end{displaymath}
\end{lema}

\bigskip

\begin{proof}
It is obvious that
$\Null(A)\supseteq \Null(A^{sym})\cap \Null(A^{skew})$.
On the other hand, if $v\in \Null(A)$ then $Av=0$, and hence $v^T A^{sym} v= v^T A v=0$.
Since $A^{sym}\le 0$, it follows that $v\in \Null(A^{sym})$.
Observing that $A^{skew}=A-A^{sym}$, we have
$v\in \Null(A^{skew})$.
Thus $v\in \Null(A^{sym})\cap \Null(A^{skew})$. We have proved that
$\Null(A)=\Null(A^{sym})\cap \Null(A^{skew})$.
Analogously, $\Null(A^T)=\Null(A^{sym})\cap \Null(A^{skew})$.
\end{proof}

\bigskip

Define 
\begin{equation}\label{equilibria}
E_{A,q}=\{\, x\in\R^n_+\,:\, A(x-q)=0\,\}
\end{equation} 
to be the affine space of equilibrium points of $X_{A,q}$.

\bigskip

\begin{teor}\label{transversal_intersection}
Given a dissipative matrix  $A$, each leaf of $\F_A$ intersects transversely 
$E_{A,q}$ in a single point.
\end{teor}

\bigskip

\begin{proof} Let $V$ be a $k\times n$ matrix whose rows form a basis of $\Null(A)^{\perp}$,
the space generated by the rows of $A$.  With this notion,
$A(x-q)=0$ is equivalent to $V(x-q)=0$.
Let $W$ be the $(n-k)\times n$ matrix used in the definition of $\F_A$,  whose rows
 form a basis of $\Null(A^T)$. With $W$ and $V$ we form the $n\times n$ matrix  $U=\left[\begin{array}{c}
W \\
\hline
V
\end{array}\right]$. By lemma \ref{kernel_equality}, $\Null(A^T)= \Null(A)$,
and this implies that  $U$ is nonsingular.
 The intersection of a leaf $g^{-1}(c)$ of $\F_A$ with the equlibria set $E_{A,q}$
is described by the system 
\begin{displaymath}
x\in g^{-1}(c)\cap E_{A,q}\quad\Leftrightarrow\quad \left \{ \begin{array}{ll}
W\log x=c\\
V(x-q)=0
\end{array} \right.\;.
\end{displaymath}
Substituting $u=\log x$, this system becomes equivalent to
\begin{displaymath}
\left \{ \begin{array}{ll}
Wu=c\\
V(e^u-q)=0 
\end{array} \right.\;.
\end{displaymath}
In order to see that each leaf of $\F_A$ intersects the equilibria set $E_{A,q}$ at a single point,
it is enough to see that 
\begin{displaymath}
{\left \{ \begin{array}{ll}
Wu=c\\
V(e^u-q)=0
\end{array} \right.} \quad \textrm{and} \quad
{\left \{ \begin{array}{ll}
Wu'=c\\
V(e^{u'}-q)=0
\end{array} \right.}
\end{displaymath}
imply  $u=u'$.
By the mean value theorem   for every  $i\in\{1,\dots,n\}$ there is some $\tilde{u}_i\in[u_i,u_i']$ such that
\begin{displaymath}
e^{u_i}-e^{u_i'}=e^{\tilde{u}_i}(u_i-u_i'),
\end{displaymath}
which in vector notation is to say that
$$ e^u-e^{u'}   =   D_{e^{\tilde{u}}}(u-u')  =   e^{\tilde{u}}*(u-u'). $$
Hence
\begin{align*}
&  {\left \{ \begin{array}{ll}
W(u-u')=0\\
V(e^u-e^{u'})=0
\end{array} \right.} \quad  \Leftrightarrow   \quad  {\left \{ \begin{array}{ll}
																					W(u-u')=0\\
																					VD_{e^{\tilde{u}}}(u-u')=0
																					\end{array} \right.}  
\quad \Leftrightarrow  \\
					           &   {\left[\begin{array}{c}
					           											W \\
					           											\hline
					           											VD_{e^{\tilde{u}}}
					           											\end{array}\right]}(u-u')=0   
\quad  \Leftrightarrow  \quad 
																	 U{\left[\begin{array}{c|c}
					           											I & 0 \\
					           											\hline
					           											0 & D_{e^{\tilde{u}}}
					           											\end{array}\right]}(u-u')=0 \;.   
\end{align*}
Therefore, as $\left[\begin{array}{c|c}
					        I & 0 \\
					        \hline
					        0 & D_{e^{\tilde{u}}}
					        \end{array}\right]$ is non-singular, we must have $u=u'$.
The  transversality follows from the nonsingularity of  $U{\left[\begin{array}{c|c}
					           											I & 0 \\
					           											\hline
					           											0 & D_{e^{\tilde{u}}}
					           											\end{array}\right]}$.
\end{proof}

\bigskip
\bigskip

\section{Stably Dissipative Systems }

\begin{defi}\label{graph}
Given a dissipative matrix $A$ we define its graph as $G_A=(\Vscr_A,\Ascr_A)$,
with  $\Vscr_A= \Vscr_{\blV} \cup \Vscr_{\whV}$, \, $\Vscr_\blV=\{\, 1\leq i\leq n \,:\, a_{ii}<0\,\}$
and
$\Vscr_\whV=\{\,1\leq i\leq n\,:\, a_{ii}=0\,\}$.
A pair $(i,j)$, with $i\neq j$, is an edge in $\Ascr_A$ \, iff\, $a_{ij}\neq 0$ or $a_{ji}\neq 0$.
\end{defi}

\bigskip

Consider a {\em simple} graph $G=(\Vscr,\Ascr)$ whose vertices are colored in black and white,
meaning there is a decomposition   $\Vscr=\Vscr_{\blV} \cup \Vscr_{\whV}$ of the vertex set.
We shall say that a vertex in $\Vscr_{\whV}$
is a {\em $\whV$-vertex}, while  a vertex in $\Vscr_{\blV}$
is a {\em $\blV$-vertex}. Such graphs will be referred as 
{\em black and white graphs}.

\bigskip

 In~\cite{RW1984} Redheffer and Walter gave the following important property of stably dissipative matrices
in terms of their associated graph.
 
\bigskip

\begin{lema}\label{teorema1RZ1981}
If $A$ is a stably dissipative matrix then every cycle of $G_A$ has at least one  
strong link $(\bbEdge)$.
\end{lema}

\begin{proof}
The proof is by contradiction, for  otherwise we could perturb $A$ into
a non dissipative matrix. \end{proof}

\bigskip

\begin{defi}\label{stably:dissipative:graph}
A  black and white graph $G$ is called {\em stably dissipative} \, iff\,
every cycle of $G$ contains at least a  strong link, i.e., an edge between $\blV$-vertices $(\bbEdge)$.
\end{defi}
The name `stably dissipative'  stems from the use we shall make of this class of graphs
 to characterize stably dissipative matrices. See Proposition \ref{sd:char} below.

\bigskip

Let us say that a dissipative matrix $ A\in\Mat_{n\times n}(\R)$ is {\em almost skew-symmetric} \, iff\,
$a_{ij}=-a_{ji}$ whenever $a_{ii}=0$ or $a_{jj}=0$, and the quadratic form
$Q(x_k)_{k\in\Vscr_\blV} = \sum_{i,j\in\Vscr_\blV} a_{ij}\,x_i\,x_j$ is negative definite.
 
\bigskip

\begin{prop}\label{diss:char} Given a dissipative matrix $ A\in\Mat_{n\times n}(\R)$,
there is a positive diagonal matrix $D$ such that $a_{ij}\,d_j=-a_{ji}\,d_i
$ whenever $a_{ii}=0$ or $a_{jj}=0$, and  for every $x_k\in\R$ with $k\in \Vscr_\blV$,\,
$\sum_{i,j\in\Vscr_\blV} a_{ij} d_i x_i x_j \leq 0$.
\end{prop}

\begin{proof} Let $D$ be a positive diagonal matrix such that for all $x\in\R^n$,
$Q(x)=\sum_{i,j} a_{ij} d_i x_i x_j \leq 0$.
Assuming $a_{ii}=0$, choose a vector $x\in\R^n$ with $x_i=0$ and $x_k=0$
for every $k\neq i,j$. Then for every $x_j\in\R$,
$$ (a_{ij} d_j + a_{ji} d_i)\,x_j + a_{jj} d_j \,x_j^2 =Q(x)\leq 0\;, $$
which implies that $a_{ij} d_j + a_{ji} d_i=0$, and everything else follows.
\end{proof}

\bigskip

Recently, Zhao and Luo \cite{ZL2010} gave a complete characterization
of stably dissipative matrices.

\bigskip

\begin{prop} \label{sd:char} Given   $ A\in\Mat_{n\times n}(\R)$, 
$A$ is stably dissipative \, iff\,  $G_A$ is a stably dissipative  
 graph  and there is a positive diagonal matrix $D$ such that $A D$ is almost skew-symmetric.
\end{prop}

\begin{proof} We outline the proof. Assuming $A$ is stably dissipative,
by lemma~\ref{teorema1RZ1981},  $G_A$ is stably dissipative.
Take a diagonal matrix $D>0$ according to lemma~\ref{lema2.1.ZL2010},
which implies that 
$Q(x_k)_{k\in\Vscr_\blV} = \sum_{i,j\in\Vscr_\blV} a_{ij} d_i x_i x_j $
is negative definite.
By proposition~\ref{diss:char}, $A D$ is almost skew-symmetric.

Conversely, assume $G_A$ is stably dissipative, assume $AD$ is almost skew-symmetric,
and take $\tilde{A}=(\tilde{a}_{ij})$ some close enough perturbation of $A$.
Let $\tilde{G}_A$ be the partial graph of $G_A$ obtained by removing every
strong link $(\bbEdge)$. Because $A$ is stably dissipative, the graph $\tilde{G}_A$ has no cycles.
We partition the vertex set $\{1,\ldots, n\}$ as follows:
Let $V_0$ be a set with an endpoint in each connected component of $\tilde{G}_A$.
Recursively, define $V_k$ to be the set of all vertices such that there is
an edge of $\tilde{G}_A$  connecting it to a vertex in $V_{k-1}$.
By construction, we have a map $i\mapsto i'$, associated with this partition, such that
$i'\in V_{k-1}$ for every $i\in V_k$ with $k\geq 1$, and $G_A$ has an edge connecting $i$ with $i'$. 
Then we consider the diagonal matrix $\tilde{D}=\mbox{diag}(\tilde{d}_j)$
whose coefficients are recursively defined by
$\tilde{d}_i=d_i$ if $i\in V_0$, and
$$ \tilde{d}_i = - \tilde{d}_{i'}\,\frac{\tilde{a}_{ii'}}{\tilde{a}_{i'i}} \quad \text{ for }\; i\in V_k\;
\text{ with }\;  k\geq 1\;. $$
It follows by induction in $k$ that  $\tilde{d}_i>0$ for every $i\in V_k$.
For $k=0$ this is automatic. Assuming this holds in $V_{k-1}$,
for any $i\in V_k$ we have $\tilde{d}_{i'}>0$, which implies that  $\tilde{d}_i>0$ because $\tilde{a}_{ii'}$ and $\tilde{a}_{i'i}$ have opposite signs.
The diagonal matrix $\tilde{D}$ is close to $D$ because $\tilde{A}$ is near $A$.
Therefore, by continuity, the quadratic form
$\tilde{Q}(x_k)_{k\in \Vscr_\blV} = \sum_{i,j\in\Vscr_\blV} \tilde{a}_{ij} \tilde{d}_i x_i x_j$ is negative definite.
On the other hand, by definition of $\tilde{d}_i$,
$\tilde{d}_i\,\tilde{a}_{ii'} + \tilde{d}_{i'}\,\tilde{a}_{i'i}=0$. 
Since every $\wwEdge$ and  $\wbEdge$ connection links $i$ with $i'$ for some $i\notin V_0$,
this implies that for every $(x_i)\in \R^n$, $ \sum_{i,j=1}^{n} \tilde{a}_{ij} \tilde{d}_i x_i x_j\leq 0$.
Hence $\tilde{A}$ is dissipative, which proves that $A$ is
stably dissipative. 
\end{proof}

\bigskip
\bigskip

\section{ Main Results }
\label{main:results}

\begin{defi} Given a stably dissipative graph $G$, we denote by $\SD(G)$  the set of all
dissipative matrices $A$ with $G_A=G$. 
\end{defi}

Our main theorem is

\begin{teor}\label{main}
Given a stably dissipative graph $G$, every matrix $A\in\SD(G)$ has the same rank.
\end{teor}

\bigskip

By this theorem we can define the {\em rank} of a stably dissipative graph $G$,
denoted hereafter by $\rank(G)$,  as the rank of any matrix in $\SD(G)$.
Together with proposition~\ref{inv:foliation}, this implies that

\bigskip

\begin{coro}\label{main_coro}
Given a stably dissipative graph $G$, for every matrix $A\in\SD(G)$,
any stably dissipative Lotka-Volterra system  with matrix $A$ has
an invariant foliation of dimension $=\rank(G)$.
\end{coro}

\bigskip

\begin{defi}
We shall say that a graph $G$ has {\em constant rank} \, iff\,
every matrix $A\in\SD(G)$ has the same rank. 
\end{defi}
With this terminology,
theorem~\ref{main} just states that every stably dissipative graph has constant rank.

\bigskip
\bigskip

\section{ Simplified Reduction Algorithm }\label{Redheffer_Reduction_Algorithm}

As before, let  $\Gamma_{A,q}$ denote the forward limit set of~(\ref{LV}),
i.e., the set of all accumulation points of  $\phi_{A,q}^t(x)$ as $t\to+\infty$.
We   say that a species $i\in\{1,\ldots,n\}$ is  {\em strongly dissipative} \, iff \,
$\Gamma_{A,q}\subset\{\,x\in\R^n_+\,:\, x_i=q_i\, \}$, or equivalently  
$\lim_{t\to+\infty} \phi^t_{A,q,i}(x)=q_i$, for all $x\in\R^n_+$.
Similarly, we  say that a species $i\in\{1,\ldots,n\}$ is  {\em weakly dissipative} \, iff \,
$\lim_{t\to+\infty} \phi^t_{A,q,i}(x)$ exists, for all $x\in\R^n_+$.
With this terminology, the algorithm of Redheffer \textit{et al.},  described  in the introduction,
is about the determination of all `strongly' and `weakly dissipative' species of the stably dissipative  system~(\ref{LV}).
Since the algorithm runs on the graph $G_A$, the conclusions drawn from the reduction procedure hold for all stably dissipative systems
that share the same graph $G_A$.

The following proposition is a slight improvement on item (b) of  theorem~\ref{Theorem_1_Redheffer_Walter}.

\begin{prop}
If $\Rscr(A)$ has only $\blV$ and $\plus$-vertices 
then the system has an invariant foliation with a single globally attractive stationary point in each leaf.
\end{prop}

\begin{proof}
Combine theorem \ref{Theorem_1_Redheffer_Walter}(b)  
with proposition~\ref{transversal_intersection}.
\end{proof}

\bigskip

In~\cite{RZ1981} Redheffer and Zhiming  make the following  remark:

\begin{rmk}\label{RZ:remark}
Let $A$ be dissipative and let every vertex $\circ$ in $G_A$ be replaced arbitrarily by $\oplus$. Then $A$ is nonsingular \, iff, \, by algebraic manipulations, every vertex can then be replaced by $\bullet$.
\end{rmk}

We shall explain this remark in terms of a simpler reduction algorithm.
Denote by $E_{A,q}$   the set of all equilibria of~(\ref{LV}), 
$E_{A,q}=\{\, x\in\R^n_+\,:\, A\,(x-q)=0\,\}$.
Let us  say that a species  $i\in\{1,\ldots,n\}$  is a {\em restriction to the equilibria} of $X_{A,q}$ \,
whenever \, $E_{A,q}\subset\{\,x\in\R^n_+\,:\, x_i=q_i\, \}$.  Notice that every strongly dissipative species is also
a restriction to the equilibria of $X_{A,q}$.
Think of colouring $i$ as black  as the statement that $i$ is a restriction to the equilibria of $X_{A,q}$.
Notice that at the begining of the reduction algorithm, described in the introduction,   the
weaker interpretation that all black vertices correspond to restrictions to the equilibria is also valid.
If we simply do not write $\plus$-vertices, but consider every $\whV$-vertex as a $\plus$-vertex,
then the reduction rules (b) and (c) can be discarded,
while the first rule, (a), becomes 
\begin{enumerate}
\item[(R)] If all neighbours of a vertex $j$ are $\blV$-vertices, except for one vertex $k$,
then we can color $k$ as a $\blV$-vertex.
\end{enumerate}

The idea implicit in the remark above is that (R) is a valid inference rule for the weaker interpretation
of the colouring statements above. Assuming that every $\whV$-vertex is a $\plus$-vertex amounts to
looking for restrictions on the equilibria set $E_{A,q}$ instead of the attractor $\Gamma_{A,q}$.
Let us still call {\em reduced graph}  to the graph, denoted by  $\Red(G)$, obtained from $G$ by successively applying rule (R) alone until
it can no longer be applied. The previous considerations show that

\begin{prop}\label{Redheffer:Equilibria}
Given a stably dissipative matrix $A$, every $\blV$-vertex of $\Red(G_A)$ is a restriction to the equilibria
of $X_{A,q}$.
\end{prop}

\bigskip

We shall write $\Red(G)=\{\blV\}$ to express that all vertices
of $\Red(G)$ are $\blV$-vertices.

\bigskip

\bigskip

\begin{coro}\label{black:reduction}
If  $G$ is a stably dissipative graph 
such that  $\Red(G)=\{\blV\}$ then every matrix $A\in\SD(G)$ is nonsingular.
In particular $G$ has constant rank.
\end{coro}

\begin{proof}
Given $A\in\SD(G)$, by Proposition~\ref{Redheffer:Equilibria} we have
$E_{A,q}=\{q\}$, which automatically implies that $A$ is nonsingular.
\end{proof}

\bigskip

In fact, the converse statement of this corollary holds by remark~\ref{RZ:remark}.
 
\begin{prop} 
Let $A$ be stably dissipative matrix. If $A$ is nonsingular then $\Red(G_A)=\{\blV\}$.
\end{prop}

\bigskip
\bigskip

\section{ Proofs }

We call any extreme $\whV$-vertex of $G$ a {\em  $\whV$-endpoint} of $G$.

\begin{lema} \label{no:endpoints}
If  a stably dissipative graph $G$
  has no $\whV$-endpoints then $\Red(G)=\{\blV\}$.
\end{lema}

\begin{proof}
Let $G$ be a stably dissipative graph with no $\whV$-endpoints.
Assume, by contradiction, that $\Red(G)\neq \{\blV\}$.
We shall construct a cycle in $\Red(G)$ with no $\bbEdge$ edges.
Since every $\whV$-vertex of $\Red(G)$ is also a $\whV$-vertex of $G$,
this will contradict the
assumption that $G$ is stably dissipative.
In the following construction we always refer to the vertex coloring of $\Red(G)$.
Take $j_0$ to be any $\whV$-vertex.  
Then, given $j_k$  take a neighbouring vertex $j_{k+1}$ to be another $\whV$-vertex,
if possible, or a $\blV$-vertex otherwise.
While the path is simple (no vertex repetitions) it can not end at some $\whV$-endpoint, and it can not contain
any $\bbEdge$ edge  because whenever we arrive to a $\blV$-vertex from a $\whV$-one we can always escape to 
another $\whV$-vertex. In fact, no $\blV$-vertex can be linked to a single $\whV$-vertex since otherwise
we could reduce it to a $\blV$-vertex by applying rule (R).
By finiteness this recursively defined path must eventually close, hence producing a cycle
with no $\bbEdge$ edges.
\end{proof}

\bigskip

Given a stably dissipative graph $G$ and some $\whV$-endpoint $i\in\Vscr_{\whV}$,
we define the {\em trimmed} graph $T_i(G)$ as follows: Let $i'\in\Vscr$ be the unique vertex
connected to $i$ by some edge of $G$. Then $T_i(G)$ is the partial graph obtained from $G$
by removing every edge incident with $i'$. The trimming operation
preserves the stable dissipativeness of graph, i.e.,
$T_i(G)$ is stably dissipative whenever $G$ is.  This follows by
definition~\ref{stably:dissipative:graph} because $T_i(G)$ is obtained by removing some edges from $G$.

\bigskip

\begin{figure}[h]
\centering{\includegraphics[width=9cm]{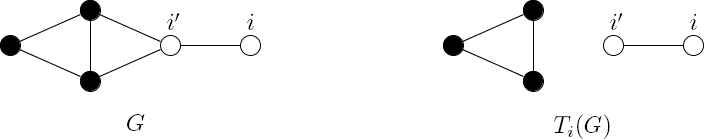}}
\caption{\footnotesize{A graph $G$ and it's trimmed graph $T_i(G)$.}}
\end{figure}

\bigskip

Similarly we define the {\em trimmed matrix} $T_i(A)$ as follows:
Annihilate every entry of row $i'$, except for $a_{i' i}$ and $a_{i' i'}$,
and annihilate every entry of column $i'$, except for $a_{ii'}$ and $a_{i' i'}$.
$$
A=\left[
\begin{array}{ccccc}
\cdot & 0 & \cdot & \ast & \cdot\\
0 & 0 & 0 & a_{ii'} & \cdot\\
\cdot & 0 & \cdot & \ast & \cdot\\
\ast & a_{i'i} & \ast & a_{i'i'} & \ast\\
\cdot & 0 & \cdot & \ast & \cdot\\
\end{array}
\right]
\qquad T_i(A)=
\left[
\begin{array}{ccccc}
\cdot & 0 & \cdot & 0 & \cdot\\
0 & 0 & 0 & a_{ii'} & \cdot\\
\cdot & 0 & \cdot & 0 & \cdot\\
0 & a_{i'i} & 0 & a_{i'i'} & 0\\
\cdot & 0 & \cdot & 0 & \cdot\\
\end{array}
\right]$$
The `$\ast$' above represent entries of $A$ that are annihilated in $T_i(A)$.

\bigskip

\begin{lema} \label{trim:matrix}
Let $i\in\Vscr_{\whV}$ be some $\whV$-endpoint of a stably dissipative graph $G$.\\
If  $A\in\SD(G)$  \, then \, $T_i(A)\in\SD(T_i(G))$\,  and \,
$\rank(T_i(A))= \rank(A)$.
\end{lema}

\begin{proof} Take $A\in \SD(G)$ and let $A'=T_i(A)$, where $i$ is some $\whV$-endpoint.
Denote, respectively, by ${\rm col}_j$ and ${\rm row}_j$  the $j^{\rm th}$ column and the $j^{\rm th}$ row of $A$,
and denote by  ${\rm col}_j'$ and ${\rm row}_j'$ the $j^{\rm th}$ column and the $j^{\rm th}$  row of the trimmed matrix.
Since $i$ is a $\whV$-endpoint, $a_{ii'}$ is the only nonzero entry in ${\rm row}_i$,
and $a_{i,i}$ is the only nonzero entry in ${\rm col}_i$.
Then the trimmed matrix $A'$ is obtained from $A$ by applying the following
Gauss elimination rules, either simultaneously  or in some arbitrary order
\begin{align*}
{\rm row}_j'&:= {\rm row}_j - \frac{a_{ji'}}{a_{i,i'}}\, {\rm row}_i\qquad j\neq i' \;,\\
{\rm col}_j'&:= {\rm col}_j - \frac{a_{i'j}}{a_{i',i}}\, {\rm col}_i\quad\qquad j\neq i'\;.
\end{align*}
Because Gauss elimination preserves the matrix rank we have $\rank(A')= \rank(A)$.
To finish the proof, it is  enough to see now that $A'$ is stably dissipative.
We use proposition~\ref{sd:char} for this purpose.
First, $G_{A'}=T_i(G)$ is stably dissipative as observed above.
Let $D$ be a positive diagonal matrix such that $A D$ is almost skew-symmetric. 
In view of  proposition~\ref{sd:char},  we only need to prove that $A'\,D$ is also almost skew-symmetric.
Notice that $G$ and $T_i(G)$ share the same black  and white vertices.
If $a'_{kk}=0$ or $a'_{jj}=0$ then also $a_{kk}=0$ or $a_{jj}=0$.
Hence, because $A D$ is almost skew-symmetric,
$a_{kj}\,d_j=-a_{jk}\,d_k$. Looking at the Gauss elimination rules  above,
we have $a_{kj}'=a_{kj}$ and $a_{jk}'=a_{jk}$, or else
 $a_{kj}'= a_{jk}'=0$. In either case we have $a_{kj}'\,d_j=-a_{jk}'\,d_k$.
Finally, we need to see that 
$Q'(x_\ell)_{\ell\in\Vscr_\blV} = \sum_{k,j\in\Vscr_\blV} a_{kj}' d_j x_k x_j$
is a  negative definite  quadratic form.
If $i'$ is a $\whV$-vertex then $Q'(x_\ell)_{\ell\in\Vscr_\blV} = \sum_{k,j\in\Vscr_\blV} a_{kj} d_j \, x_k x_j$
is negative definite because $A D$ is almost skew-symmetric.
Otherwise, if $i'$ is a $\blV$-vertex, given a nonzero vector $(x_\ell)_{\ell\in\Vscr_\blV}$
we define $(x_\ell')_{\ell\in\Vscr_\blV}$ letting $x_\ell'=x_\ell$ for $\ell\neq i'$,
while $x_{i'}'=0$.  Then
$$ Q'(x_\ell)_{\ell\in\Vscr_\blV} =  \underbrace{a_{i'i'}}_{<0} d_{i'} x_{i'}^2 + 
\underbrace{ \sum_{k,j\in\Vscr_\blV} a_{kj} d_j x_k' x_j'}_{=Q(x_\ell')_{\ell\in\Vscr_\blV}\leq 0}  <0\;,$$
since $(x_\ell)_{\ell\in\Vscr_\blV}\neq 0$ implies that 
either $x_{i'}\neq 0$ or else $(x_\ell')_{\ell\in\Vscr_\blV}\neq 0$.
This proves that $Q'$ is negative definite.
\end{proof}

\bigskip

As a simple corollary of the previous lemma we obtain

\begin{lema}[Trimming lemma]\label{trimming}
Let $i\in\Vscr_{\whV}$ be some $\whV$-endpoint of a stably dissipative graph $G$.
If $T_i(G)$ has constant rank then so has $G$, and
$\rank(G)=\rank(T_i(G))$.
\end{lema}

\bigskip

We can now prove theorem~\ref{main}.
\begin{proof} 
Define recursively a sequence of graphs
$G_0, G_1,\ldots, G_m$, with $G_0=G$, and where $G_{i+1}=T_{j_i}(G_i)$ for some
$\whV$-endpoint $j_i$ of $G_i$. This sequence will end at some graph $G_m$ with no $\whV$-endpoint.
By Lemma~\ref{no:endpoints}  we have $\Red(G_m)=\{\blV\}$.
The connected components of $G_m$ are either reducible to $\blV$-vertices by iteration of rule (R),
or else composed by $\whV$-vertices alone.
Since the $\whV$-components can not be trimmed anymore, they must be either formed of a single
$\whV$-vertex, or else a single $\wwEdge$ edge.
By corollary~\ref{black:reduction}, $G_m$ has constant rank. 
Finally, applying inductively Lemma~\ref{trimming}
we see that all graphs $G_i$ have constant rank. Hence $G$, in particular, has
constant rank.
\end{proof}

\bigskip

\begin{rmk}
The previous proof gives a simple recipe to compute the rank of a graph.
Trim $G$ while possible. In the end, discard the single $\whV$-vertex components
and count the remaining vertices.
\end{rmk}


\begin{center}
\begin{table}[h]
\caption{\footnotesize{Some graph trimming examples.}}
\begin{tabular}[c]{ccccc}
Original graph & & Trimmed graph & & Graph rank \\
\hline \\
\vspace{2mm}
\includegraphics[width=2cm]{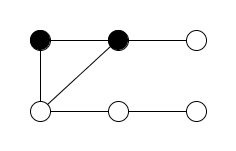}
& &
\includegraphics[width=2cm]{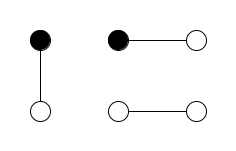}
& & 
6 \\
\vspace{2mm}
\includegraphics[width=2cm]{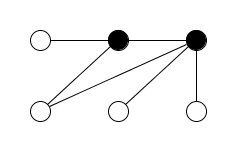}
& &
\includegraphics[width=2cm]{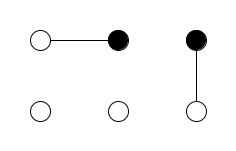}
& & 
4 \\
\vspace{2mm}
\includegraphics[width=2cm]{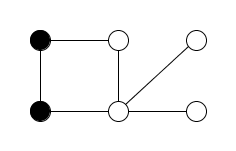}
& &
\includegraphics[width=2cm]{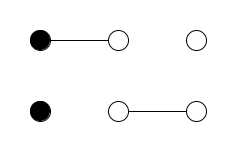}
& & 
5 \\
\hline \\
\end{tabular}
\end{table}
\end{center}

\bigskip
\bigskip

\section{Trimming Effect on Dynamics}

In this last section we use an example to describe the effect of trimming a stably dissipative
matrix on the underlying dynamics.

Consider the system
\begin{displaymath}
E=\left \{ \begin{array}{ll}
								\dot{x}_1=x_1((x_2-1)+(x_7-1)) \\
								\dot{x}_2=x_2(-2(x_1-1)+(x_3-1)) \\
								\dot{x}_3=-x_3(x_2-1) \\
								\dot{x}_4=x_4((x_5-1)-(x_7-1)) \\
								\dot{x}_5=x_5(-2(x_4-1)+(x_6-1)) \\
								\dot{x}_6=-x_6(x_5-1) \\
								\dot{x}_7=x_7(-(x_1-1)+(x_4-1)-(x_7-1)),
								\end{array} \right.
\end{displaymath}
with interaction matrix
\begin{displaymath}
A={\left[\begin{array}{ccc|ccc|c}
0 & 1 & 0 & 0 & 0 & 0 & 1 \\
-2 & 0 & 1 & 0 & 0 & 0 & 0 \\
0 & -1 & 0 & 0 & 0 & 0 & 0 \\
\hline
0 & 0 & 0 & 0 & 1 & 0 & -1 \\
0 & 0 & 0 & -2 & 0 & 1 & 0 \\
0 & 0 & 0 & 0 & -1 & 0 & 0 \\
\hline
-1 & 0 & 0 & 1 & 0 & 0 & -1
\end{array}\right]}
\end{displaymath}
and fixed point $q=(1,1,1,1,1,1,1,)$. $E$ has associated graph $G_A$ represented in figure \ref{Graph_6_1}.\\

\begin{figure}[h]
\centering{\includegraphics[width=6cm]{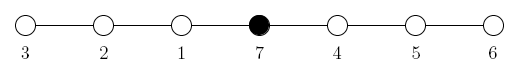}}
\caption{\footnotesize{Associated graph of matrix $A$, $G(A)$.}\label{Graph_6_1}}
\end{figure}


\begin{figure}[h]
\centering{\includegraphics[width=6cm]{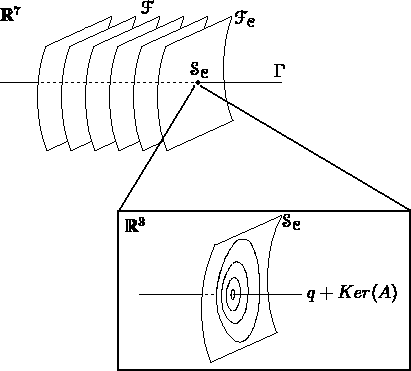}}
\caption{\footnotesize{Phase portrait of a system $E$.}\label{Invariant_foliation}}
\end{figure}

The null space, $\Null(A)$, is generated by the vector $(1,0,2,1,0,2,0)$. Hence the foliation $\F$, with leaves $\F_c$ given by
\begin{displaymath}
\F_c=\{x\in\R^7 \, :\, \log x_1+2\log x_3+\log x_4+2\log x_6=c \}\,,
\end{displaymath}
is an invariant foliation with dimension $6$ in $\R^7$. The system's phase portrait is represented in figure \ref{Invariant_foliation}, being the atractor a $3$-plan transversal to $\F$ given by
\begin{displaymath}
\Gamma=\{x\in\R^7 \, :\, x_1=x_4,\,x_2=x_5,\,x_3=x_6,\,x_7=1 \}\,.
\end{displaymath}

The intersection of each leaf $\F_c$ with $\Gamma$ is a surface $\Scal_c$ given by
\begin{displaymath}
\Scal_c=\F_c\cap\Gamma=\{(x_1,x_2,x_3,x_1,x_2,x_3,1) \, :\, \log x_1+2\log x_3=\frac{c}{2} \},
\end{displaymath}
wich is foliated into invariant curves by the level sets of $h$, defined in (\ref{Liapunov_function}). Note that $\Scal_c$ corresponds to an invariant leaf of the conservative system with graph $\circ \put(0.5,2.5){\line(1,0){10}} \quad \circ \put(0.5,2.5){\line(1,0){10}} \quad \circ$.\\ 

With the first trim on $G$ we get the graph $T_6(G)$ represented in figure \ref{Graph_6_1_trim_1}.

\begin{figure}[h]
\centering{\includegraphics[width=6cm]{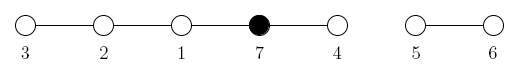}}
\caption{\footnotesize{The trimmed graph of $G$, $T_6(G)$.}\label{Graph_6_1_trim_1}}
\end{figure}

This corresponds to annihilate the entries $(4,5)$ and $(5,4)$ of the original matrix $A$. Notice that the components $x_5$ and $x_6$ of the system are independent of the rest. Hence the dynamics of this new system is the product of two independent LV systems represented in figure~\ref{Trim_1_dynamic}.

\begin{figure}[h]
\centering{\includegraphics[width=7cm]{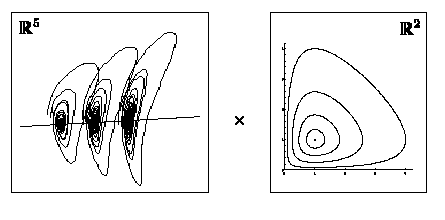}}
\caption{\footnotesize{Representation of the dynamic of the system $E_1$.}\label{Trim_1_dynamic}}
\end{figure}

The five dimensional system on the left of figure~\ref{Trim_1_dynamic} has a straight line of equilibria. Moreover it leaves invariant a foliation of dimension four with a single globally attractive fixed point on each leaf. The right-hand side system is a typical conservative predator-prey.

Now we have two different possibilities of trimming the graph $T_6(G)$: we can either choose the $\whV$-endpoint $3$ or else $4$. In the first case we get the graph $T_3(T_6(G))$ represented in figure \ref{Graph_6_1_trim_2_1}, whose dynamics is illustrated in figure~\ref{Trim_2_1_dynamic}. 

\begin{figure}[h]
\centering{\includegraphics[width=6cm]{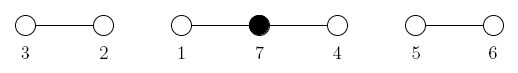}}
\caption{\footnotesize{The trimmed graph $T_3(T_6(G))$ of $T_6(G)$.}\label{Graph_6_1_trim_2_1}}
\end{figure}

\begin{figure}[h]
\centering{\includegraphics[width=9cm]{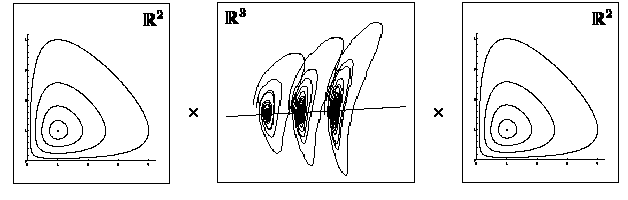}}
\caption{\footnotesize{Representation of the system's dynamics associated to the graph $T_3(T_6(G))$.}\label{Trim_2_1_dynamic}}
\end{figure}

The three dimensional system in the middle of figure~\ref{Trim_2_1_dynamic} has a straight line of equilibria. Moreover it leaves invariant a foliation of dimension two with a single globally attractive fixed point on each leaf. The left and right-hand side systems are typical conservative  predator-preys.\\

In the second case we get the graph $T_4(T_6(G))$ represented in figure \ref{Graph_6_1_trim_2_2}, whose dynamics is depicted in figure~\ref{Trim_2_2_dynamic}.

\begin{figure}[h]
\centering{\includegraphics[width=6cm]{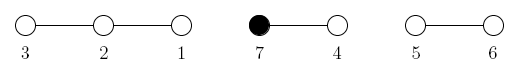}}
\caption{\footnotesize{The trimmed graph $T_4(T_6(G))$ of $T_6(G)$.}\label{Graph_6_1_trim_2_2}}
\end{figure}

\begin{figure}[h]
\centering{\includegraphics[width=9cm]{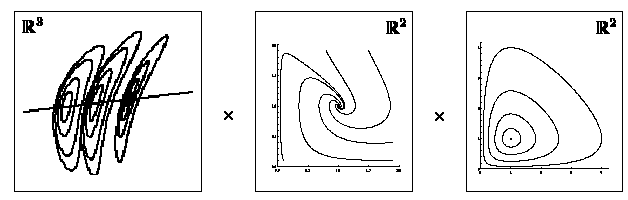}}
\caption{\footnotesize{Representation of the system's dynamics associated to the graph $T_4(T_6(G))$.}\label{Trim_2_2_dynamic}}
\end{figure}

Here, the left-hand side three dimensional system is conservative, leaving invariant a foliation of dimension two transversal to a straight line of equilibria. The middle and right-hand side systems are typical predator-prey, respectively, dissipative and conservative.\\

Trimming $T_4(T_6(G))$ choosing the $\whV$-endpoint $1$ we get the graph $T_1(T_4(T_6(G)))$ represented in figure \ref{Graph_6_1_trim_3}, whose dynamics is a product of three predator-prey systems, illustrated in figure~\ref{Trim_3_dynamic}, with a one dimensional system consisting of equilibria.\\

Notice that by trimming $T_3(T_6(G))$ we obtain an isomorphic graph to the one in figure \ref{Graph_6_1_trim_3}.

\begin{figure}[h]
\centering{\includegraphics[width=6cm]{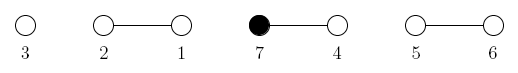}}
\caption{\footnotesize{The trimmed graph $T_1(T_4(T_6(G)))$ of $T_4(T_6(G))$.}\label{Graph_6_1_trim_3}}
\end{figure}

\begin{figure}[h]
\centering{\includegraphics[width=9cm]{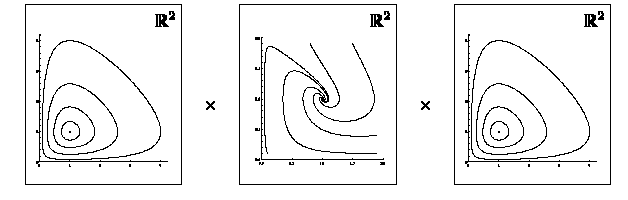}}
\caption{\footnotesize{Representation of the system's dynamics associated to the graph $T_1(T_4(T_6(G)))$.}\label{Trim_3_dynamic}}
\end{figure}

\bigskip
\bigskip
\bigskip

\section*{Acknowledgments}

The first author was partially supported by Fundação para a Ciência e
Tecnologia through the project "Randomness in Deterministic Dynamical
Systems and Applications" (PTDC/MAT/105448/2008).

\bigskip
\bigskip
\bigskip

\bibliographystyle{plain}

\bigskip
\bigskip

\end{document}